\documentclass[12pt]{article}

\usepackage{amssymb,amsmath,amsfonts,amssymb}
\usepackage{graphics,graphicx,color,hhline}
\usepackage{bbm} 
\usepackage{euscript}  
\usepackage{authblk} 
 \usepackage{mathtools}
\def\@abssec#1{\vspace{.05in}\footnotesize \parindent .2in 
{\bf #1. }\ignorespaces} 
\graphicspath{{/EPSF/}{Figures/}}   
\newtheorem{theorem}{Theorem}[section]
\newtheorem{lemma}[theorem]{Lemma}

\newtheorem{corollary}[theorem]{Corollary} 

\newtheorem{remark}[theorem]{Remark}

\def \Rm {\mathbb R}
\def \Nm {\mathbb N}

\def\II{\mathbb{I}}
\newcommand{\eps}{\varepsilon}

\newcommand{\ds}{\displaystyle}

\newcommand{\calL}{\mathcal L}

\newcommand{\calF}{\mathcal F}

\newcommand{\calE}{\mathcal E}
\newcommand{\Tr}{\textrm{Tr}}

\newcommand{\calS}{\mathcal S}
\newcommand{\calJ}{\mathcal J}

\newcommand{\calD}{\mathcal D}

\def\fref#1{{\rm (\ref{#1})}}

\newcommand{\cout}[1]{}

\def\un{{\mathbbmss{1}}}

\newcommand{\be}{\begin{equation}}
\newcommand{\ee}{\end{equation}}
\newcommand{\bea}{\begin{eqnarray}}
\newcommand{\eea}{\end{eqnarray}}
\newcommand{\bee}{\begin{eqnarray*}}
\newcommand{\eee}{\end{eqnarray*}}
\newcommand{\bal}{\begin{align*}}
\newcommand{\eal}{\end{align*}}

\DeclarePairedDelimiter\bra{\langle}{\rvert}
\DeclarePairedDelimiter\ket{\lvert}{\rangle}
\begin{document}
{\title{Local smoothing for the quantum Liouville equation}}

 \author{Olivier Pinaud \footnote{pinaud@math.colostate.edu}}
 \affil{Department of Mathematics, Colorado State University, Fort Collins CO, 80523}

\maketitle

\begin{abstract}
 We analyze in this work the regularity properties of the density operator solution to the quantum Liouville equation. As was recently done for the Strichartz inequalities, we extend to systems of orthonormal functions the local smoothing estimates satisfied by the solutions to the Schr\"odinger equation. We show in particular that the local density associated to the solution to the free, linear, quantum Liouville equation admits locally fractional derivatives of given order provided the data belong to some Schatten spaces.
\end{abstract}
%\tableofcontents

\section{Introduction}
The quantum Liouville equation (also referred to as the von Neumann equation) describes the evolution of a (possibly infinite) statistical ensemble of quantum particles. The main object of interest is the \textit{density operator}, denoted by $u$ in the sequel, which is in general a trace class, self-adjoint and nonnegative operator on some Hilbert space (here $L^2(\Rm^d)$, where $d\geq 1$ is dimension). With physical constants set to one, the free, linear Liouville equation reads
\be \label{liou}
\left\{
\begin{array}{l}
i \partial_t u=\big[-\Delta,u \big]+\varrho,\\
u(t=0)=u_0,
\end{array}
\right.
\ee
where $u_0$ and $\varrho$ are given operators, $[\cdot, \cdot]$ denotes commutator between operators and $\Delta$ is the usual Laplacian on $\Rm^d$. Roughly speaking, \fref{liou} can be seen as a large (or infinite) system of coupled Schr\"odinger equations. It is then an interesting problem to understand how the dispersive and regularity properties of the solutions to the Schr\"odinger equation translate to the solutions to the Liouville equation. Indeed, setting $\varrho=0$ for the sake of concreteness, the solution $u$ reads
$$
u(t)=\sum_{j \in \Nm} \lambda_j \ket{e^{i t\Delta}u_j} \bra{e^{it \Delta} u_j} \qquad \textrm{if} \qquad u_0 =\sum_{j \in \Nm} \lambda_j \ket{u_j} \bra{u_j},
$$
where the $\lambda_j$ are here positive numbers, the $u_j$ are orthonormal in $L^2(\Rm^d)$, and we used the Dirac notation for the rank one projector $\varphi \mapsto u_j (u_j,\varphi) \equiv \ket{u_j} \langle u_j \vert\varphi \rangle$. The mathematical properties of the operator $e^{it \Delta}$ are well known, and the question is to understand how they translate to the collection defined by $u$. This has been an extensive subject of research, and some of the most notable results are the following: 
\begin{itemize}
\item[-] \textit{The Lieb-Thirring inequalities \cite{LT,LP}.} They are the density operator counterpart of Gagliardo-Nirenberg-Sobolev inequalities: for $u_0$ as before, if we define formally the local density $n_{u_0}$ and the local kinetic energy $\calE_{u_0}$ by
$$
n_{u_0}:=\sum_{j \in \Nm} \lambda_j |u_j|^2, \qquad \calE_{u_0}:=\sum_{j \in \Nm} \lambda_j |\nabla u_j|^2,
$$
an example of such inequalities is 
$$
\|n_{u_0}\|_{L^q(\Rm^d)} \leq C \left(\sum_{j \in \Nm} \lambda_j^p \right)^{\theta/p} \|\calE_{u_0}\|^{1-\theta}_{L^1(\Rm^d)}
$$
where $p \in [1,\infty]$ and
$$
\theta=\frac{2p}{(d+2)p-d}, \qquad q=\frac{(d+2)p-d}{dp-(d-2)}.
$$
\item[-] \textit{The Strichartz inequalities \cite{LewStri,FrankSabin}.} They are generalizations to orthonormal systems of the classical Strichartz estimates for the Schr\"odinger equation, and read, see \cite{FrankSabin}, Theorem 8,
$$
\left\|\sum_{j \in \Nm} \lambda_j |e^{i t\Delta}u_j|^2 \right\|_{L^p_tL^q_x(\Rm\times \Rm^d)} \leq \left( \sum_{j \in \Nm} \lambda_j^{\frac{2q}{q+1}}\right)^{\frac{q+1}{2q}},$$
with $p,q,d \geq 1$,
$$
\frac{2}{p}+\frac{d}{q}=d, \qquad 1 \leq q <1+\frac{2}{d-1}.
$$
\end{itemize}

These estimates on orthonormal systems have some remarkable properties: as mentioned in \cite{LewStri,FrankSabin}, if $\lambda_j=0$ for $j>N$, then they behave much better in terms of $N$ than the estimates obtained by applying the triangle inequality and the scalar estimates for the Schr\"odinger group $e^{i t\Delta}$. This can be seen as a consequence of the orthogonality of the $u_j$, while that latter property is not used with the triangle inequality. Another interesting feature is the following: standard ways to define rigorously the local density $n_{u_0}$ are either to assume that $u_0$ is trace class (i.e. $\sum_{j \in \Nm} \lambda_j < \infty$), or to assume that $u_0$ enjoys some smoothness, e.g. 
\be \label{regu0}\Tr \left((\II-\Delta)^{\beta/2} (u_0)^2 (\II-\Delta)^{\beta/2}\right)^p < \infty
\ee 
for appropriate $p \geq 1$ and $\beta\geq 0$, see e.g. \cite{sabin1} (above $\Tr$ denotes the operator trace). The Strichartz estimates offer much weaker conditions that justify the definition of the local density: the local density of the solution $u$ to the Liouville equation for $\varrho=0$ is defined in the space $L^p_tL^q_x$ as soon as $u_0$ lies in Schatten space $\calJ_{2q/(q+1)}$. This can be seen as a combined effect of the orthogonality of the $u_j$ and the dispersive effects of the Schr\"odinger operator. The situation is similar when $\varrho \neq 0$.\\

Our motivation in this work is to pursue the analysis of the properties of the solutions to the Liouville equation in the light of those of the Schr\"odinger equation. We are interested in the \textit{local smoothing effect} discovered by Constantin and Saut \cite{ConstSaut}, Sjolin \cite{sjolin} and Vega \cite{vega}: if $v_0 \in L^2(\Rm^d)$, then $v=e^{it\Delta}v_0$ admits locally spatial fractional derivatives of order $1/2$. The real interest in this result is the gain in differentiability, and not the gain in integrability obtained in turn by using Sobolev embeddings since the Strichartz estimates provide better exponents. When $u_0$ and $\varrho$ are in some Schatten spaces of order $p\geq 1$, we show that the local density $n_u$ admits locally spatial fractional derivatives of order $\beta$, where $\beta$ depends on the dimension $d$ and $p$. In particular, $\beta$ decreases as $p$ increases. Note that here again, the local density $n_u$ can be defined without a trace class assumption or the regularity \fref{regu0} on the data: this is a combined effect of the orthogonality and of the local gain in derivatives; the latter provides local estimates similar to \fref{regu0} on the solution $u$, which in turn allow one to define $n_u$. We also obtain results on the trace of the operator $\chi (\II-\Delta)^{\beta/2} u (\II-\Delta)^{\beta/2} \chi$ ($\chi$ is a smooth cut-off) in terms of the Schatten norms of $u_0$ and $\varrho$. 

To conclude this introduction, we would like to mention that Lieb-Thirring inequalities, Strichartz inequalities and local smoothing estimates are important in the analysis of non-linear quantum systems, see e.g. \cite{sabin1, sabin2}.

The paper is organized as follows: we introduce some notation in section \ref{secresults}, recall one of the main theorems of Constantin and Saut, and state our main results in Theorem \ref{th1} and Corollary \ref{cor1}. We present as well in Lemma \ref{pr1} a generalization of the results of Constantin and Saut that is the main ingredient in the proof of Theorem \ref{th1}. Sections \ref{proofth}, \ref{prooflem} and \ref{proofcor} are devoted to the proofs of Theorem \ref{th1}, Lemma \ref{pr1} and Corollary \ref{cor1}.\\
% $\EuScript{I}$
% \begin{itemize}
% \item Estimer $D^\alpha n$
% \item comparer avec Strichartz
% \item g\'erer le cas homog\`ene
% \item esimations en fonction de $N$ comme Lewin (non, \`a cause du $\beta$, juste dire que l'ineg triang traite le cas $p=1$)
% \item en dire plus sur le cas $\varrho$ pas toujours positif
% \item mettre le resultat sur la trace de $\varrho$
% \item dire que l'on ne peut pas g\'erer le cas inhomog\`ene avec l'in\'egalit\'e triangulaire (encore n'importe quoi...)
% \end{itemize}

\noindent \textbf{Acknowledgment.} This work is supported by NSF CAREER grant DMS-1452349. The author would like to thank the anonymous referee for the suggested corrections.

\section{Main results} \label{secresults}
We start with some notation.

\paragraph{Notation.} We denote by $\calF_x$ the Fourier transform with respect to the $x \in \Rm^d$ variable, and by $\calF$ the Fourier transform with respect to $(t,x)$, with the convention
$$
\calF_x \varphi(\xi)=\int_{\Rm^d} e^{-i x \cdot \xi} \, \varphi(x) dx.
$$ For $\beta \in \Rm$, let $\Lambda_\beta:=(\II-\Delta)^{\beta/2}$, and let $D^\beta:=(-\Delta)^{\beta/2}$ be the fractional Laplacian, that we will only consider for $\beta \in (0,1/2]$. The operator $\Lambda_\beta$ and $D^\beta$ are defined in the Fourier space as pseudo-differential operators with symbols $(1+|\xi|^2)^{\beta/2}$ and $|\xi|^{\beta}$. We systematically use $r'$ for the conjugate exponent of $r \in [1,\infty]$ defined as usual by $\frac{1}{r}+\frac{1}{r'}=1$.
For $\beta \in \Rm$ and $p\in[1,\infty]$, we will use the potential space $H^{\beta,p}(\Rm^d)$ defined by $H^{\beta,p}(\Rm^d)=\{ \varphi \in \calS'(\Rm^d), \Lambda_\beta \varphi \in L^p(\Rm^d) \}$ with norm $\| \cdot \|_{H^{\beta,p}(\Rm^d)}=\| \Lambda_\beta \cdot \|_{L^p(\Rm^d)}$, where $\calS'(\Rm^d)$ is the space of tempered distributions on $\Rm^d$. We denote as well by $\calS(\Rm^d)$ the Schwartz space and by $C_0^\infty(\Rm^d)$ the space of infinitely differentiable functions on $\Rm^d$ with compact support. The space of bounded operators on $L^2(\Rm^d)$ is denoted by $\calL(L^2(\Rm^d))$ and its associated norm is $\| \cdot \|$. The Schatten spaces on $L^2(\Rm^d)$ are denoted by $\calJ_p$, $p \in [1,\infty]$ with the convention $\calJ_\infty=\calL(L^2(\Rm^d))$.

We decompose the solution $u$ to \fref{liou} into homogeneous and non-homogeneous parts as  $u=\Upsilon+\Gamma$ with
\be \label{defop}
\left\{
\begin{array}{ll} \ds \Upsilon(t):=e^{i t \Delta} \,u_0\, e^{-i t\Delta}, &\qquad \ds \Gamma(t):=\int_0^t e^{i (t-s) \Delta} \varrho(s) e^{-i (t-s) \Delta} ds\\
 \ds \sigma_\beta(t) := \chi_t\, \Lambda_\beta\, \Upsilon(t)\,  \Lambda_\beta\,  \chi_t, &\qquad \ds \gamma_\beta(t) := \chi_t\, \Lambda_\beta\, \Gamma(t)\,  \Lambda_\beta\,  \chi_t,
\end{array}
\right.
\ee
where $\chi_t \equiv \chi(t,x) \in C_0^\infty(\Rm^{d+1})$, $u_0 \in \calJ_p$, $\varrho \in L^1(\Rm,\calJ_p)$, and $u_0,\varrho$ are self-adjoint operators. We do not assume here that $u_0$ and $\varrho$ have a particular sign. It is then a classical matter that $u$ is the unique mild solution to \fref{liou} and belongs to $C^0(\Rm,\calJ_p)$. Assuming that $\Upsilon, \Gamma \in C^0(\Rm,\calJ_p)$ admit the decompositions
$$
\Upsilon=\sum_{j \in \Nm} \lambda_j \ket{\psi_j} \bra{\psi_j}, \qquad \Gamma=\sum_{j \in \Nm} \mu_j \ket{\phi_j} \bra{\phi_j},
$$
where $\lambda_j$ and $\mu_j$ are real and $(\psi_i,\psi_j)_{L^2(\Rm^d)}=(\phi_i,\phi_j)_{L^2(\Rm^d)}=\delta_{ij}$, we define formally the local densities of $\sigma_\beta$ and $\gamma_\beta$ by
\be \label{defn}
n_{\sigma_\beta}= \sum_{j \in \Nm} \lambda_j | \chi_t  \Lambda_\beta \psi_j |^2, \qquad n_{\gamma_\beta}= \sum_{j \in \Nm} \mu_j | \chi_t  \Lambda_\beta \phi_j|^2.
\ee
For a trace class operator $u$, the local density $n_u$ is defined rigorously by duality by
$$
\int_{\Rm^d} n_u(x) \varphi(x) dx=\Tr \big( u \varphi \big), \qquad \forall \varphi \in L^\infty(\Rm^d),
$$
where the multiplication operator by $\varphi$ and the function $\varphi$ are identified (this will be systematically done in the sequel).

Before stating our main results, we recall the important estimates of Constantin and Saut \cite{ConstSaut}: 
% With
% $$
% \varrho(t)=\sum_{p \in \Nm} \rho_p(t) \ket{e_p(t)} \bra{e_p(t)}
% $$
% we have
% $$
% n_{\gamma_\beta(t)}= \sum_{p \in \Nm} \int_0^t \rho_p(s) | \chi_t \Lambda_\beta e^{i (t-s) \Delta} e_p(s) |^2 ds.
% $$
% $$
% \int_{\Rm} \int_{\Rm^d} n_{\gamma_\beta(t)}(t,x) \varphi(t,x) dt dx = \int_{\Rm} \Tr \big( \gamma_{\beta(t)} \varphi_t \big) dt
% $$
\begin{theorem} \label{thCS} Let $q,r' \in [1,2]$, $\beta>0$ such that
$$
\left\{
\begin{array}{l} \ds
\beta<\frac{1}{r}-d\left(\frac{1}{q}-\frac{1}{r}\right), \quad \textrm{if}\qquad r\geq 2, \quad r>q\\
\beta \leq \frac{1}{2}\quad \textrm{if}\qquad r=q=2.
\end{array}
\right.
$$ 
Then, for $\chi_t \equiv \chi \in C_0^\infty(\Rm^{d+1})$, $\varphi_t \equiv \varphi \in \calS(\Rm^{d+1})$, the following estimate holds
$$
\left(\int_{\Rm^d} (1+|\xi|^2)^{\beta q /2} |\calF(\chi_t \varphi_t)(-|\xi|^2,\xi)|^q d\xi \right)^{1/q} \leq C_\chi \| \varphi_t\|_{L^{r'}(\Rm^{d+1})}.
$$ 
The constant $C_\chi$ above depends on $\chi$ but not on $\varphi$.
\end{theorem}
\begin{remark} \label{remdual} We will also use the dual form of Theorem \ref{thCS}. Under the same hypotheses as above, we have
$$
\left(\int_{\Rm^{d+1}} |\chi(t,x)|^r |\Lambda_\beta e^{it \Delta}\varphi(x) |^r dtdx \right)^{1/r} \leq C_\chi \|\varphi\|_{L^{q}(\Rm^{d})},
$$ 
which shows the local gain of a derivative of order $\beta$. 
\end{remark}
\begin{remark}\label{remchi}
It is known that the condition $\chi \in C^\infty_0(\Rm^{d+1})$ is not necessary in Theorem \ref{thCS}. In particular, when $q=r=2$, the function $\chi(t,x)=(1+|x|)^{-s}$, for $s>1/2$ is sufficient, see e.g. \cite{vega}. A close look at the proof of Theorem \ref{thCS} shows that in the general case where $r\geq 2$, $r>q$, the previous function is also sufficient. We will use that observation in the proof of our main result.
\end{remark}

Our first result concerns the local densities $n_{\sigma_\beta}$ and $n_{\gamma_\beta}$, that we show are properly defined in some $L^q(\Rm^{d+1})$ space provided the data are in appropriate $\calJ_p$ spaces. We also show that the operators $\sigma_\beta$ and $\gamma_\beta$ are trace class. 
\begin{theorem}\label{th1}
Let $n_{\sigma_\beta}$ and $n_{\gamma_\beta}$ be the local densities defined formally in \fref{defn} with $u_0 \in \calJ_p$ and $\varrho \in L^1(\Rm,\calJ_p)$ for $p \in [1,\frac{2d}{2d-1})$. Then, $n_{\sigma_\beta}$ and $n_{\gamma_\beta}$ belong to $L^{r/2}(\Rm^{d+1})$, and we have the estimates
\be \label{estnth}
\|n_{\sigma_\beta}\|_{L^{r/2}(\Rm^{d+1})} \leq C_\chi \|u_0\|_{\calJ_p}, \qquad \|n_{\gamma_\beta}\|_{L^{r/2}(\Rm^{d+1})} \leq C_\chi \|\varrho\|_{L^1(\Rm,\calJ_p)},
\ee
for $r \in [2,\infty)$ and  $\beta >0$ such that 
$$
\left\{
\begin{array}{l} \ds
\ds \beta < \frac{1+d}{r}-\frac{d}{2}, \qquad \textrm{if} \qquad p=1 \qquad \textrm{and} \qquad r>2\\[2mm]
\ds \beta \leq \frac{1}{2}, \qquad \textrm{if} \qquad p=1 \qquad \textrm{and} \qquad r=2\\[2mm]
\ds \beta < \frac{1+d}{r}-\frac{d}{2}-\frac{d}{p'}, \qquad \textrm{if} \qquad  p>1.
\end{array}
\right.
$$
Moreover, under the same conditions as above with $r=2$, we have the following estimates on the operators $\sigma_\beta$ and $\gamma_\beta$ defined in \fref{defop}:
\be \label{esttraceth}
%$$
\|\sigma_\beta\|_{L^1(\Rm,\calJ_1)} \leq C_\chi \|u_0\|_{\calJ_p}, \qquad \|\gamma_\beta\|_{L^1(\Rm,\calJ_1)} \leq C_\chi \|\varrho\|_{L^1(\Rm,\calJ_p)}.
%$$
%\Tr \int_\Rm \sigma_\beta(t) dt \leq C \|\varrho_0\|_{\calJ_p}, \qquad \Tr \int_\Rm \gamma_\beta(t) dt \leq C \|\varrho\|_{L^1(\Rm_+,\calJ_p)}
\ee
The constant $C_\chi$ in \fref{estnth}-\fref{esttraceth} depends on $\chi$ but not on $u_0$ and $\varrho$.
\end{theorem}

Note that for $\beta$ to be positive,  the following conditions need to be satisfied:
$$
r\leq 2+\frac{2}{d}, \qquad p <\frac{2d}{2d-1}, \qquad \frac{1}{p'}<\frac{1+d}{rd}-\frac{1}{2}.
$$
Besides, as can be expected, the gain in differentiability is maximal when $p=1$, that is when the strongest assumption is made on the data. Whether Theorem \ref{th1} is optimal or not is an open question. We use next Theorem \ref{th1} to obtain a local gain of fractional derivatives for the local densities $n_\Upsilon$ and $n_\Gamma$ of the operators $\Upsilon$ and $\Gamma$.
\begin{corollary} \label{cor1} Let $u_0 \in \calJ_p$ and $\varrho \in L^1(\Rm, \calJ_p)$ for $p \in [1,\frac{2d}{2d-1})$ and $\beta > 0$ such that 
$$
\left\{
\begin{array}{l} \ds
\beta \leq \frac{1}{2}, \qquad \textrm{if} \qquad p=1\\[2mm]
\ds \beta < \frac{1}{2}-\frac{d}{p'}, \qquad \textrm{if} \qquad  p>1.
\end{array}
\right. 
$$ Then, for any $T>0$, any $v \in C_0^\infty(\Rm^d)$ and $\chi=|v|^2$, $D^\beta(\chi n_{\Upsilon})$ and $D^\beta(\chi n_{\Gamma})$ belong to $L^1(-T,T,L^{\frac{d}{d-\beta}}(\Rm^d))$, with the following estimates,
$$
\| D^\beta(\chi n_{\Upsilon}) \|_{L^1(-T,T,L^{\frac{d}{d-\beta}}(\Rm^d))}  \leq C_T \|u_0\|_{\calJ_p}, \quad \| D^\beta(\chi n_{\Gamma})\|_{L^1(-T,T,L^{\frac{d}{d-\beta}}(\Rm^d))} \leq C_T \|\varrho\|_{L^1(\Rm,\calJ_p)}.
$$
The constant $C_T$ above depends on $T$ and $v$ but not on $u_0$ and $\varrho$.
\end{corollary}
\begin{remark} \label{intime} In Corollary \ref{cor1}, we only have an $L^1$ integrability in time since Theorem \ref{th1} was only exploited in the case $r=2$. The reason for this is explained in the proof of the corollary.% and extending the above results to a higher integrability in time is an interesting question. 
\end{remark}
% \begin{remark} The results of Theorem \ref{th1} and Corollary \ref{cor1} can readily be generalized to operators $\varrho_0$ and $\varrho$ with no particular sign. For $\varrho \in \calJ_p$, $p<\infty$, decomposing $\varrho$ into positive and negative  parts as $\varrho=\varrho_+-\varrho_-$, $\varrho_{\pm} \geq 0$, $\varrho_+ \varrho_-=\varrho_- \varrho_+=0$, the linearity of the trace yields
% $$
% |n_{\gamma_\beta}|=|n_{\gamma_{+,\beta}}-n_{\gamma_{-,\beta}}| \leq |n_{\gamma_{+,\beta}}|+|n_{\gamma_{-,\beta}}|,
% $$
% and it suffices to to apply \fref{estnth} to $n_{\gamma_{\pm,\beta}}$ together with
% $$
% \| \varrho_{\pm} \|_{\calJ_p} \leq \|\varrho\|_{\calJ_p}, \qquad p \in [1,\infty).
% $$
%  The same observation holds for $n_{\sigma_\beta}$, \fref{esttraceth} and the estimates of the Corollary \ref{cor1}.
% \end{remark}
The core of the proof of Theorem \ref{th1} is the following extension of Theorem \ref{thCS}. As will be clear further, the weight $|x|^\alpha+1$ is introduced in order to handle the case when $u_0$ and $\varrho$ are not trace class, but only in some $\calJ_p$ with $p>1$.
\begin{lemma} \label{pr1} Let $\phi_t \equiv \phi(t,x) \in \calS(\Rm^{d+1})$ and $\chi_t \equiv \chi(t,x) \in C_0^\infty(\Rm^{d+1})$. Then,  we have the following estimate:
$$
\left\| (|x|^\alpha+1) \Lambda_{\alpha+\beta} \int_{\Rm } dt \, e^{-i t \Delta} \left(\chi_t \phi_t\right) \right\|_{L^2(\Rm^d)} \leq C_\chi \| \phi_t \|_{L^{r'}(\Rm^{d+1})},
$$
for $\alpha \in (0,\frac{1}{4}]$, $\beta \geq 0$, and $ r \in [2,\infty)$, such that 
\be \label{conslem}
\left\{\begin{array}{l}
\ds 2\alpha+\beta < \frac{(d+1)}{r}-\frac{d}{2}, \qquad \textrm{if} \qquad r\in (2,\infty),\\
\ds 2\alpha+\beta \leq \frac{1}{2}, \qquad \textrm{if} \qquad r =2.
\end{array}
\right.
\ee
Moreover, the constant $C_\chi$ above depends on $\chi$ but not on $\phi$.
\end{lemma}
Above, the bound $\alpha \leq 1/4$ is arbitrary, the value $1/4$ is chosen since the largest value of the r.h.s. of \fref{conslem} is $1/2$. In order to see the link with Theorem \ref{thCS}, note that
$$
\calF(\chi_t \varphi_t)(-|\xi|^2,\xi) = \calF_x \left(\int_{\Rm } dt e^{-i t \Delta} \left(\chi_t \phi_t\right) \right)(\xi).
$$
The result of Lemma \ref{pr1} is fairly intuitive: for the sake of simplicity, replace $|x|^\alpha$ by $x_1$ (the first coordinate of $x$); then, up to a commutator between $x_1$ and $\Lambda_{\alpha+\beta}$, we use the classical fact that
$$
\Lambda_{\alpha+\beta} \, x_1\, e^{-it\Delta} = \Lambda_{\alpha+\beta}\, e^{-it\Delta} \, x_1+ 2 i t\Lambda_{\alpha+\beta}\, \partial_{x_1} \, e^{-it\Delta}.
$$
The weights $x_1$ and $t $ in both terms of the r.h.s are handled by the fact that $\chi_t$ has a compact support, and the second term exhibits a loss of a derivative of order one (of order $\alpha$ in the Lemma, which explains \fref{conslem}). Of course, what makes Lemma \ref{pr1} non trivial is the fact that $\alpha$ is not an integer, which necessitates then the use of fractional derivatives for which there are no simple product and chain rules. In particular, the fractional derivatives are not local, which makes the use of Theorem \ref{thCS} not direct. We conclude this section by the following remark. 

\begin{remark} \label{dual2} The dual version of Lemma \ref{pr1} reads, under the same assumptions on the parameters,
$$
\left(\int_{\Rm^{d+1}} |\chi(t,x)|^r | \Lambda_{\alpha+\beta} e^{it \Delta}\left((|\cdot|^\alpha+1)\varphi\right)(x) |^r dtdx \right)^{1/r} \leq C_\chi \|\varphi\|_{L^{2}(\Rm^{d})}.
$$ 
\end{remark}
\section{Proof of Theorem \ref{th1}} \label{proofth}
%We remark first that without lack of generality we can limit ourselves in the proof to operators $\varrho_0$ and $\varrho$ positive and finite rank. Any self-adjoint operator $\varrho$ can indeed be decomposed into $\varrho=\varrho_+-\varrho_-$, with $\varrho_{\pm} \geq 0$, $\varrho_{+} \varrho_{-}=\varrho_{-} \varrho_+=0$, and if $\varrho \in \calJ_p$,
%$$
%\| \varrho_{\pm} \|_{\calJ_p} \leq \|\varrho\|_{\calJ_p}, \qquad p \in [1,\infty).
%$$

The proof is an application of Theorem \ref{thCS} and Lemma \ref{pr1}. We focus on the term $\gamma_\beta$ and explain at the end of the proof the straighforward adaptation to $\sigma_\beta$. We remark first that we can limit ourselves to nonnegative operators $\varrho$ in $\calJ_p$, $p<\infty$. We can indeed decompose $\varrho$ as $\varrho=\varrho_+-\varrho_-$, $\varrho_{\pm} \geq 0$, $\varrho_+ \varrho_-=\varrho_- \varrho_+=0$, and  the linearity of the trace yields
$$
|n_{\gamma_\beta}|=|n_{\gamma_{+,\beta}}-n_{\gamma_{-,\beta}}| \leq n_{\gamma_{+,\beta}}+n_{\gamma_{-,\beta}}.
$$
It suffices then to apply the estimates that we will derive for  $n_{\gamma_{\pm,\beta}}$, together with
$$
\| \varrho_{\pm} \|_{\calJ_p} \leq \|\varrho\|_{\calJ_p}, \qquad p \in [1,\infty).
$$
Also, in order to justify some formal computations, we consider  ``smooth'' operators $\varrho$ satisfying the following conditions: $\varrho$ is finite rank and 
\be \label{condrho}
\int_\Rm dt\, \Tr \Big((1+|x|^k) (-\Delta)^n \,\varrho(t)\, (-\Delta)^n (1+|x|^k)\Big) <\infty, \quad \forall k,n \in \Nm.
\ee
The final estimates are then obtained by density since such operators are dense in $L^1(\Rm,\calJ_p)$, $p<\infty$: on the one hand finite rank operators are dense in $\calJ_p$, $p<\infty$, and on the other, we have, for any smooth cut-off $\chi_\eta$ with $\chi_\eta(x) \to 1$ pointwise as $\eta \to 0$,
$$
\chi_\eta  e^{\eta \Delta}\,  \varrho \, e^{\eta \Delta} \chi_\eta  \to \varrho  \qquad \textrm{in} \quad L^1(\Rm_+,\calJ_p) \qquad \textrm{as} \qquad \eta \to 0, 
$$
%$$
%\int_0^T dt \| (\II-\eta \Delta)^{-\beta/2} \varrho(t) (\II-\eta \Delta)^{-\beta/2}-\varrho(t)\|_{\calJ_p} \to 0 \qquad \textrm{as} \qquad \eta \to 0, 
%$$
whenever $\varrho \in L^1(\Rm,\calJ_p)$, $p\in [1,\infty]$. The regularity of $\varrho$ is then propagated to the operator $\Gamma(t)=\int_0^t e^{i (t-s) \Delta} \varrho(s) e^{-i (t-s) \Delta} ds$, see e.g. \cite{cazenave}, section 2.5, and as a consequence $\Gamma$ verifies
%\be \label{reg}
$$
\sup_{t \in \Rm} \Tr \Big((1+|x|^k) (-\Delta)^n \,\Gamma(t)\, (-\Delta)^n (1+|x|^k)\Big) <\infty, \quad \forall k,n \in \Nm.
$$
%\ee
In particular, if $(\mu_j,\phi_j)_{j\in \Nm}$ is the spectral decomposition of $\Gamma$, this shows that
$$
\sup_{t \in \Rm} \sum_{j \in \Nm} \mu_j(t) \left\|(1+|x|^k) (-\Delta)^n \,\phi_j(t)\right\|^2_{L^2(\Rm^d)} < \infty, \quad \forall k,n \in \Nm,
$$
and therefore $(\mu_j(t))^{1/2}\phi_j(t) \in \calS(\Rm^d)$ for all $t \in \Rm$.\\

The proof starts with the following lemma:
\begin{lemma} \label{expdens} Let $\beta \in [0,1/2]$. For any $\varphi\in C_0^\infty(\Rm^{d+1})$, the following relation holds:
$$
\int_{\Rm} \int_{\Rm^d}  dt dx\, n_{\gamma_\beta}(t,x) \varphi(t,x) =\int_\Rm ds \Tr \, \big(  e^{-i s \Delta} \varrho(s) e^{is \Delta} F_\beta(s) \big),
$$
where
$$
F_\beta(s)=\int_{|s|}^\infty dt \, \Lambda_{\beta}e^{-i t \Delta }f_t e^{i t \Delta } \Lambda_{\beta} \in C^0(\Rm,\calL(L^2(\Rm^d))), \qquad f_t=\chi_t \varphi_t \chi_t.
$$
\end{lemma}
\begin{proof}
We begin with the following equalities, supposing that the support in $t$ of $\chi_t$ is included  on $[-T,T]$:
\bee
\int_{\Rm} \int_{\Rm^d}  dt dx \, n_{\gamma_\beta(t)}(t,x) \varphi(t,x) &=& \int_{-T}^T dt \, \Tr \big( \gamma_\beta(t) \varphi_t \big)\\
&=&\int_{-T}^T dt \,\Tr \left( \chi_t \Lambda_{\beta} \int_0^t ds e^{i(t- s) \Delta} \varrho(s) e^{-i (t-s) \Delta }\Lambda_{\beta} \chi_t \, \varphi_t ds \right)\\
&=&\int_{-T}^T \int_0^t ds dt \, \Tr \Big( \chi_t \Lambda_{\beta}e^{i(t- s) \Delta} \varrho(s)e^{-i (t-s) \Delta }\Lambda_{\beta}\chi_t \, \varphi_t  \Big)\\[2mm]
&:=&I.
\eee
The last equality is justified by 
\begin{align} \label{justi}
&\int_0^t ds \,\left|\Tr \left( \chi_t \Lambda_{\beta}e^{i(t- s) \Delta} \varrho(s)e^{-i (t-s) \Delta }\Lambda_{\beta}\chi_t \, \varphi_t  \right)\right| \\ \nonumber
&\hspace{2cm}=\int_0^t ds \,\left|\Tr \left( \chi_t e^{i(t- s) \Delta} \Lambda_{\beta} \varrho(s)\Lambda_{\beta} e^{-i (t-s) \Delta }\chi_t \, \varphi_t  \right)\right| \\ \nonumber
&\hspace{2cm} \leq \|\chi_t\|^2_{L^\infty(\Rm^{d+1})} \|\varphi_t \|_{L^\infty(\Rm^{d+1})} \int_\Rm ds\, \| \Lambda_{\beta} \varrho(s) \Lambda_{\beta} \|_{\calJ_1}  < \infty.
\end{align}
Above, we used the fact that $\Lambda_\beta$ and $e^{it \Delta}$ commute. Since the operator $\Lambda_\beta$ is not bounded, we need an additional step to justify the use of the cyclicity of the trace and to obtain that 
$$
I=\int_{-T}^T ds\,\Tr \left( e^{-i s \Delta} \varrho(s)e^{i s \Delta }  \Lambda_{\beta} \int_{|s|}^\infty dt  e^{-i t \Delta} f_t e^{i t \Delta} \Lambda_{\beta} \right),
$$
which would end the proof. We therefore regularize $I$ as
$$
I_\eta=\int_{-T}^T \int_0^tds dt\, \Tr \left( \chi_t \Lambda^\eta \Lambda_{\beta} e^{i(t- s) \Delta} \varrho(s)e^{-i (t-s) \Delta }\Lambda_{\beta} \Lambda^\eta \chi_t \, \varphi_t  \right),
$$
where $\Lambda^\eta=(\II-\eta \Delta)^{-1}$. Similar estimates as in \fref{justi} show that $I_\eta \to I$ as $\eta \to 0$. Using the semigroup property of $e^{it \Delta}$, and the fact that $\Lambda^\eta \Lambda_\beta$ is bounded, we can write
\bee
I_\eta &=&\int_{-T}^T  \int_0^t ds dt\, \Tr \left(e^{-i s \Delta} \varrho(s) e^{is \Delta } e^{-i t \Delta}\Lambda_{\beta} \Lambda^\eta f_t \Lambda^\eta \Lambda_{\beta} e^{i t \Delta}\right) \\
&=&\int_{-T}^T ds  \int_{|s|}^\infty  dt \,\Tr \left(e^{-i s \Delta} \varrho(s) e^{is \Delta } \Lambda_{\beta} \Lambda^\eta e^{-i t \Delta} f_t e^{i t \Delta}  \Lambda^\eta \Lambda_{\beta} \right)\\
&=&\int_{-T}^T ds\, \Tr \big(  e^{-i s \Delta} \varrho(s) e^{is \Delta} F^\eta_\beta(s) \big),
\eee
where, for all $s\in \Rm$, 
$$
F^\eta_\beta(s):=\int_{|s|}^\infty dt  \Lambda_{\beta}\Lambda^\eta e^{-i t \Delta }f_t e^{i t \Delta } \Lambda^\eta \Lambda_{\beta}\in \calL(L^2(\Rm^d)).
$$
Above, the use of the Fubini theorem and the inversion of the integral and the trace are justified by similar estimates as \fref{justi}. We now show that $F^\eta_\beta \to F_\beta$ in $C^0([-T,T],\calL(L^2(\Rm^d)))$ as $\eta \to 0$. We remark first that $F_\beta(s) \in \calL(L^2(\Rm^d))$ for all $s \in [-T,T]$ and $\beta \in [0,1/2]$ thanks to Theorem \ref{thCS} and Remark \ref{remdual}. Indeed, for any $v \in C_0^\infty(\Rm^d)$,
\bee
\| F_\beta(s) v \|_{L^2(\Rm^d)}&=&\left\| \Lambda_{\beta} \int_\Rm dt  e^{-i t \Delta } \chi_t g_t \right\|_{L^2(\Rm^d)}, \qquad g_t:=\un_{(|s|,\infty)} \varphi_t \chi_t  \Lambda_{\beta}  e^{i t \Delta } v,\\
&\leq & C \|g_t \|_{L^2(\Rm^{d+1})}\\
& \leq & C  \|\varphi_t\|_{L^\infty(\Rm^{d+1})} \|\chi_t  \Lambda_{\beta}  e^{i t \Delta } v\|_{L^2(\Rm^{d+1})}\\
&\leq & C \|v\|_{L^2(\Rm^d)}.
\eee
The continuity of $F_\beta(s)$ in  $s$ is straightforward. We then write
$$
F_\beta^\eta-F_\beta=(\Lambda^\eta-\II) F_\beta  \Lambda^\eta+F_\beta  (\Lambda^\eta-\II),
$$
and conclude with the fact that $(\Lambda^\eta-\II)\to 0$ in $\calL(L^2(\Rm^{d}))$ together with $F_\beta \in C^0([-T,T],\calL(L^2(\Rm^d)))$. This ends the proof of the lemma since $F_\beta(s)=0$ for $|s| \geq T$.
\end{proof}

\paragraph{The case $\varrho$ trace class.} From Lemma \ref{expdens}, we have
\bee
\left|\int_{\Rm} \int_{\Rm^d} n_{\gamma_\beta}(t,x) \varphi(t,x) dt dx \right| &\leq& \sup_{s \in \Rm} \| F_\beta(s) \| \int_\Rm ds \, \Tr \big(  e^{i s \Delta} \varrho(s) e^{- is \Delta} \big)\\
&=&\sup_{s \in \Rm} \| F_\beta(s) \| \|\varrho\|_{L^1(\Rm,\calJ_1)},
\eee
and it remains to estimate $F_\beta$ in terms of $\varphi$. As in the proof of Lemma \ref{expdens}, we write
\bee
\| F_\beta(s) v \|_{L^2(\Rm^d)}&=&\left\| \Lambda_{\beta} \int_\Rm dt  e^{-i t \Delta } \chi_t g_t \right\|_{L^2(\Rm^d)}, \qquad g_t=\un_{(|s|,\infty)} \varphi_t \chi_t  \Lambda_{\beta}  e^{i t \Delta } v,\\
&\leq & C \|g_t \|_{L^{r'}(\Rm^{d+1})}\\
& \leq & C  \|\varphi_t\|_{L^{r'q'}(\Rm^{d+1})} \|\chi_t  \Lambda_{\beta}  e^{i t \Delta } v\|_{L^{r'q}(\Rm^{d+1})}\\
&\leq & C \|\varphi_t\|_{L^{r'q'}(\Rm^{d+1})} \|v\|_{L^2(\Rm^d)}.
\eee
The first inequality above is an application of Theorem \ref{thCS} with parameters satisfying
\be \label{eqq1}
\left\{\begin{array}{l}
\ds \beta < \frac{1+d}{r}-\frac{d}{2} \qquad \textrm{if} \qquad 1 \leq r' <2, \qquad \\
\ds \beta\leq \frac{1}{2} \qquad \textrm{if} \qquad r'=2.
\end{array} \right.
\ee
The second inequality follows from the H\"older inequality with $1\leq q \leq \infty$, and the last inequality is a consequence of the dual version of Theorem \ref{thCS} given in Remark \ref{remdual}, provided we have the relations
\be \label{eqq2}
\left\{\begin{array}{l}
\ds \beta < \frac{1+d}{r' q}-\frac{d}{2} \qquad \textrm{if} \qquad 2< r'q <\infty, \qquad \\
\ds \beta\leq \frac{1}{2} \qquad \textrm{if} \qquad r'q=2.
\end{array} \right.
\ee
% \be \label{eqq2}
% 2 \leq r'q <\infty, \qquad \beta \leq \frac{(1+d)}{r'q}-\frac{d}{2}.
% \ee
We optimize the right-hand-side of the inequalities of \fref{eqq1}-\fref{eqq2} by setting $q=1$ when $r'=2$, and when $r'<2$, we set $r'q=r$, i.e. $q=r-1$. In this case, $r'q'=r/(r-2)$ and $r'q'/(r'q'-1)=r/2$, which yields by duality the estimate
$$
\|n_{\gamma_\beta}\|_{L^{r/2}(\Rm^{d+1})} \leq C \|\varrho\|_{L^1(\Rm,\calJ_1)}, \qquad \textrm{with} \qquad \left\{\begin{array}{l}
\ds \beta < \frac{1+d}{r}-\frac{d}{2} \qquad \textrm{if} \qquad 1\leq r'<2, \qquad \\
\ds \beta\leq \frac{1}{2} \qquad \textrm{if} \qquad r'=2.
\end{array} \right.
$$
Note that for $\beta$ to be positive, we need the condition $r < 2(d+1)/d$.

\paragraph{The case $\varrho \in \calJ_p$, $p>1$.} This case is more interesting since new estimates are required. We start as before with Lemma \ref{expdens}. Since $\varrho$ is not trace class, one may expect a lower value for the exponent $\beta$. A simple way to proceed is to compensate for the fact that $\varrho \notin \calJ_1$ by introducing the operator $J_\alpha:=\Lambda_\alpha (1+|x|^\alpha)$ for an appropriate $\alpha$ in 
\bee
\Tr \, \big(  e^{-i s \Delta} \varrho(s) e^{is \Delta} F_\beta(s) \big)&=&\Tr \, \left( J_\alpha J_\alpha^{-1} e^{-i s \Delta} \varrho(s) e^{is \Delta} (J_\alpha^*)^{-1}J_\alpha^*  F_\beta(s) \right):=II(s).
\eee
As in the proof of Lemma \ref{expdens}, we cannot use directly the cyclicity of the trace since $J_\alpha$ is not bounded. We therefore introduce a regularization as before, with the difference that we now use Lemma \ref{pr1} instead of Theorem \ref{thCS} to pass to the limit. This justifies the fact that
\bea \nonumber
|II(s)|&=&\left|\Tr \, \left( J_\alpha^{-1} e^{-i s \Delta} \varrho(s) e^{is \Delta} (J_\alpha^*)^{-1}J_\alpha^*  F_\beta(s) J_\alpha \right) \right|\\
&\leq & \Tr \, \left( J_\alpha^{-1} e^{-i s \Delta} \varrho(s) e^{is \Delta} (J_\alpha^*)^{-1} \right) \| J_\alpha^*  F_\beta(s) J_\alpha \|. \label{T2}
\eea
Using the H\"older inequality for $\calJ_p$ spaces, we find
$$
\Tr \, \left( J_\alpha^{-1} e^{-i s \Delta} \varrho(s) e^{is \Delta} (J_\alpha^*)^{-1} \right) \leq \|\varrho(s)\|_{\calJ_p} \| J_\alpha^{-1} \|^2_{\calJ_{2p'}} = \|\varrho(s)\|_{\calJ_p} \| (1+|x|^\alpha)^{-1} \Lambda_{-\alpha}\|^2_{\calJ_{2p'}}.
$$
Above, we used the facts that $(J_\alpha^*)^{-1}=(J_\alpha^{-1})^{*}$ and that the $\calJ_p$ norms of $A$ and $A^*$ are equal. The Kato-Seiler-Simon inequality \cite{Simon-trace}, Chapter 4, then yields
$$
\| (1+|x|^\alpha)^{-1} \Lambda_{-\alpha}\|_{\calJ_{2p'}} \leq C \|(1+|x|^\alpha)^{-1}\|_{L^{2p'}(\Rm^d)} \|(1+|x|^2)^{-\alpha/2}\|_{L^{2p'}(\Rm^d)}, 
$$
which is finite whenever $2\alpha p'>d$. It remains to treat the term in \fref{T2} involving $F_\beta$. It is treated in the same fashion as in case $p=1$, with the difference  that we use Lemma \ref{pr1} and its dual version instead of Theorem \ref{thCS}. Hence, with
$$
g_t=\un_{(|s|,\infty)} \varphi_t \chi_t  \Lambda_{\alpha+\beta}  e^{i t \Delta }  (1+|x|^\alpha) v,
$$
we find the estimate
\bee
\| J_\alpha^* F_\beta(s) J_\alpha v \|_{L^2(\Rm^d)}&=&\left\| (1+|x|^\alpha)\Lambda_{\alpha+\beta} \int_\Rm dt  e^{-i t \Delta } \chi_t g_t \right\|_{L^2(\Rm^d)}\\
&\leq & C \|\varphi_t\|_{L^{r'q'}(\Rm^{d+1})} \|v\|_{L^2(\Rm^d)},
\eee
under the conditions
%\be \label{eqqq1}
$$
2\alpha+\beta < \frac{1+d}{r}-\frac{d}{2}, \qquad \textrm{for} \qquad1 < r' < 2, \qquad \textrm{and} \qquad 2\alpha+\beta \leq \frac{1}{2}, \qquad \textrm{for} \qquad r'=2.
$$
%\ee
Owing to $2\alpha p'>d$, the latter condition becomes the one of Theorem \ref{th1}. Going back to Lemma \ref{expdens}, this finally yields by duality
$$
\|n_{\gamma_\beta}\|_{L^{r/2}(\Rm^{d+1})} \leq C \|\varrho\|_{L^1(\Rm,\calJ_p)}, \qquad 1 < r' \leq 2, \qquad \beta < \frac{1+d}{r}-\frac{d}{2}-\frac{d}{p'}.
$$
Note that for $\beta$ to  be positive, we need the conditions
$$
r\leq 2+\frac{2}{d}, \qquad p'>2d, \qquad \frac{1}{p'}<\frac{d+1}{rd}-\frac{1}{2}.
$$
\begin{remark} It is interesting to explore another route to the above estimate on $n_{\gamma_\beta}$. Instead of introducing the operators $J_\alpha$, the term $II$ could be directly bounded by
$$
\|\varrho\|_{\calJ_p} \|F_\beta(s)\|_{\calJ_{p'}}, \qquad p>1.
$$
The main difficulty is now to estimate $F_\beta$ in the Schatten space $\calJ_{p'}$. This does not seem direct, in particular it is unclear whether the method of \cite{LewStri}, which is used to derive Strichartz estimates and consists in writing $\Tr((F_\beta)^{p'})$ (when $p'$ is an integer) as a multiple integral in time and in switching trace and integration, can be employed or not since the order of integration matters for estimating $F_\beta$.
\end{remark}
\paragraph{Estimates \fref{esttraceth} and conclusion.} We address now estimates \fref{esttraceth}, which are direct according to what was previously done: adapting Lemma \ref{expdens} yields
$$
\|\gamma_\beta\|_{L^1(\Rm,\calJ_1)} = \int_\Rm dt\, \Tr \big( \gamma_\beta(t)\big)= \int_\Rm dt\, \Tr \big( \varrho(t) F_\beta(t) \big),
$$
where $F_\beta$ is the same as in Lemma \ref{expdens} with now $\varphi_t\equiv 1$. Introducing the operator $J_\alpha$ as before, we find
$$
\|\gamma_\beta\|_{L^1(\Rm_+,\calJ_1)} \leq C \sup_{t\in \Rm} \| J_\alpha^* F_\beta(t) J_\alpha\| \int_\Rm dt\, \| \varrho(t)\|_{\calJ_p}, 
$$
and we already know that the term involving $F_\beta$ is finite provided $\beta$ satisfies the conditions of Theorem \ref{th1} with $r=2$. A similar method yields the estimate on $\sigma_\beta$.

This concludes the proof of the estimates involving $\gamma_\beta$. The estimates on  $n_{\sigma_\beta}$ are derived with the relation
$$
\int_{\Rm} \int_{\Rm^d} n_{\sigma_\beta}(t,x) \varphi(t,x) dt dx = \Tr \big( u_0 F_\beta(0) \big),% \qquad \varphi=0 \quad \textrm{for} \quad t\leq 0,
$$
and estimating $F_\beta$ as before. This ends the proof of Theorem \ref{th1}.
\section{Proof of Lemma  \ref{pr1}} \label{prooflem}
Let $u_t:=\chi_t  \phi_t \in C_0^\infty(\Rm^{d+1})$ and define
$$
I:=(|x|^\alpha+1) \Lambda_{\alpha+\beta} \int_{\Rm } dt e^{-i t \Delta} u_t.
$$
We then write, for $[A,B]=AB-BA$ the usual commutator between operators,
\bee
I&=&\Lambda_{\alpha+\beta} (|x|^\alpha+1) \int_{\Rm } dt e^{-i t \Delta} u_t+ \big[|x|^\alpha+1,\Lambda_{\alpha+\beta}\big]  \int_{\Rm } dt e^{-i t \Delta} u_t\\
&:=&I_0+I_1+I_2,
\eee
where
$$
I_0=\Lambda_{\alpha+\beta} |x|^\alpha \int_{\Rm } dt e^{-i t \Delta} u_t, \qquad I_1=\Lambda_{\alpha+\beta} \int_{\Rm } dt e^{-i t \Delta} u_t,
$$
and $I_2$ is the term involving the commutator. We start with $I_0$, the most singular term. 
\paragraph{The term $I_0$.}
For $\varphi$ smooth and $\alpha>0$, let 
$$
A^\alpha_t \varphi(x):=2^\alpha t^\alpha  e^{i\frac{|x|^2}{4t}} D^\alpha \left( e^{-i\frac{|x|^2}{4t}} \varphi(x)\right).
$$
Then, according to \cite{ponce}, Proposition 3, we have the following commutation relation between $|x|^\alpha$ and the operator $e^{-it \Delta}$:
$$
|x|^\alpha e^{-i t \Delta} \varphi= e^{-i t \Delta} A^\alpha_t \varphi.
$$
This leads to the decomposition of $I_0$ below:
\bee
I_0&=&2^\alpha \Lambda_{\alpha+\beta}\int_{\Rm } dt e^{-i t \Delta}\left[  t^\alpha  e^{i\frac{|x|^2}{4t}} D^\alpha \left( e^{-i\frac{|x|^2}{4t}} u_t \right) \right]\\
&:=&J_1+J_2,
\eee
with
\begin{align*}
&J_1=2^\alpha \Lambda_{\alpha+\beta}\int_{\Rm } dt e^{-i t \Delta}D^\alpha \left( t^\alpha u_t \right)\\
&J_2=2^\alpha \Lambda_{\alpha+\beta}\int_{\Rm } dt \,t^\alpha   e^{-i t \Delta}\left[  e^{i\frac{|x|^2}{4t}} D^\alpha \left( e^{-i\frac{|x|^2}{4t}} u_t \right)-D^\alpha \left( u_t \right) \right].
\end{align*}
The term $J_1$ is directly estimated with Theorem \ref{thCS}: since $D^\alpha$ and $e^{-it \Delta}$ commute, we have
$$
\| J_1\|_{L^2(\Rm^d)} \leq C \left\|  \Lambda_{2 \alpha+\beta}\int_{\Rm } dt  e^{-i t \Delta} \, t^\alpha u_t \right\|_{L^2(\Rm^d)},
$$
and therefore, applying Theorem \ref{thCS} with $q=2$,  and supposing that the support of $\chi_t$ in the $t$ variable is included in $[-T,T]$, the following estimate holds
\be \label{estJ1}
\| J_1\|_{L^2(\Rm^d)} \leq C \|t^\alpha \phi_t\|_{L^{r'}((-T,T)\times \Rm^{d})} \leq C \|\phi_t\|_{L^{r'}(\Rm^{d+1})},
\ee
with
\be \label{mainconst}
\left\{
\begin{array}{l}
\ds 2\alpha+\beta < \frac{(d+1)}{r}-\frac{d}{2}, \qquad r'\in [1,2).\\
\ds 2\alpha+\beta \leq \frac{1}{2}, \qquad r'=2.
\end{array}
\right.
\ee
This treats the leading term and yields the condition on $\alpha$ and $\beta$ stated in the lemma. We turn now to the term $J_2$. For $\varphi \in C^1_c(\Rm^{d})$, we use the following representation formula for the fractional Laplacian (see e.g. \cite{stein1}): we have, for any $\alpha \in (0,1)$,
\be \label{locfrac}
 D^\alpha(\varphi)(x)=c_{d,\alpha} \int_{\Rm^d} \frac{\varphi(x)-\varphi(y)}{|x-y|^{d+\alpha}} dy.
\ee
Note that there is no need to introduce a principal value in the definition since $\alpha \in (0,1)$ and $\varphi \in C^1$. We then write
\bee
e^{i\frac{|x|^2}{4t}} D^\alpha \left(  e^{-i\frac{|x|^2}{4t}} t^\alpha u_t  \right)-D^\alpha \left( t^\alpha u_t \right)&=&c_{d,\alpha} \int_{\Rm^d} \frac{t^\alpha u_t(x)- e^{i\frac{|x|^2-|y|^2}{4t}} t^\alpha u_t(y)}{|x-y|^{d+\alpha}} dy-D^\alpha \left( t^\alpha u_t \right)\\
&=&c_{d,\alpha} \int_{\Rm^d} \frac{(1- e^{i\frac{|x|^2-|y|^2}{4t}}) t^\alpha u_t(y)}{|x-y|^{d+\alpha}} dy\\[3mm]
&:=&R_{t,0}.
\eee
We need now to estimate a term of the form
$$
\Lambda_{\alpha+\beta}\int_{\Rm } dt   e^{-i t \Delta} R_{t,0}.
$$
For this, we cannot directly apply Theorem \ref{thCS} since the function $R_{t,0}$ does not have the required regularity and compact support. This is therefore where we invoke Remark \ref{remchi} and the fact that the function $\chi(t,x)=(1+|x|)^{-s}$ for $s>1/2$ is sufficient. We actually choose $(1+|x|)^{-1}$ for simplicity. Then,
\bee
(1+|x|) R_{t,0}
&=&c_{d,\alpha} \int_{\Rm^d} (1+|y|)\frac{(1- e^{i\frac{|x|^2-|y|^2}{4t}}) t^\alpha u_t(y)}{|x-y|^{d+\alpha}} dy\\
&&+c_{d,\alpha} \int_{\Rm^d} (|x|-|y|)\frac{(1- e^{i\frac{|x|^2-|y|^2}{4t}}) t^\alpha u_t(y)}{|x-y|^{d+\alpha}} dy\\
&:=&R_{t,1}+R_{t,2}.
\eee
This yields
$$
J_2=2^\alpha \Lambda_{\alpha+\beta}\int_{\Rm } dt e^{-i t \Delta}(1+|x|)^{-1}\left[R_{t,1}+R_{t,2}\right],
$$
and as a consequence of Theorem \ref{thCS},
\be \label{estJ2}
\| J_2\|_{L^2(\Rm^d)} \leq C \|R_{t,1}+R_{t,2}\|_{L^{q'}(\Rm^{d+1})},\; \textrm{with} \; \left\{
\begin{array}{l}
\ds \alpha+\beta < \frac{(d+1)}{q}-\frac{d}{2}, \quad q'\in [1,2).\\
\ds \alpha+\beta \leq \frac{1}{2}, \quad q'=2.
\end{array}
\right.%\alpha+\beta \leq \frac{(d+1)}{q}-\frac{d}{2}, \qquad q'\in[1,2].
\ee
We estimate now $R_{t,1}+R_{t,2}$. We will actually only treat the most singular term $R_{t,1}$ since $R_{t,2}$ follows from the same techniques. We will use for this the simple inequality below, for $a \in [0,1]$, 
\be \label{AF}
\left| 1- e^{i\frac{|x|^2-|y|^2}{4t}}\right| \leq 2^{1-3a}t^{-a} |x+y|^a|x-y|^a\leq 2^{1-3a}t^{-a}\left( 2^a|y|^a|x-y|^{a} +|x-y|^{2a}\right).
\ee
Using \fref{AF} in $R_{t,1}$ with $a=\alpha+\eps$, for some $\eps>0$ such that $\eps \leq 1-\alpha$, we find 
\bee
|R_{t,1}| &\leq& C t^{-\eps} \int_{\Rm^d} (1+|y|^{1+\alpha+\eps})\frac{|u_t(y)|}{|x-y|^{d-\eps}} dy+C t^{-\eps} \int_{\Rm^d} (1+|y|)\frac{|u_t(y)|}{|x-y|^{d-\alpha-2\eps}} dy\\
&:=&R_{t,3}+R_{t,4}.
\eee
We start with the most singular term $R_{t,3}$. For any $r'\in (1,2]$ as in \fref{estJ1}, we claim that we can find $q'$ (the one in \fref{estJ2}) and $\eps>0$ such that
\be \label{R3}
\|R_{t,3}\|_{L^{q'}(\Rm^{d+1})} \leq C \|\phi_t\|_{L^{r'}(\Rm^{d+1})}.
\ee
We will use for this the following classical estimate on Riesz potentials, see e.g. the Hardy-Littlewood-Sobolev inequality of \cite{RS-80-2}, section IX.4,
\be \label{riesz}
\left\| \int_{\Rm^d} \frac{\varphi(y)}{|\cdot-y|^{d-\eps}} dy \right\|_{L^{\frac{s d}{d-\eps s}}(\Rm^d)} \leq C \|\varphi\|_{L^s(\Rm^d)}, \qquad 1<s<\infty, \qquad 1<\frac{s d}{d-\eps s}<\infty.
\ee
Fix now $r'\in (1,2]$ as in \fref{estJ1}, and set  $q'=\frac{s d}{d-\eps s}$ in \fref{estJ2} for some $s$ such that $1<s<r'$. Then, for \fref{estJ2} (in the case $q'<2$) and \fref{riesz} to hold, it is required that
\be \label{const}
1<\frac{s d}{d-\eps s} < 2, \qquad \alpha+\beta < \left(\frac{1}{s'}+\frac{\eps}{d}\right)(d+1)-\frac{d}{2}, \qquad r<s'<\infty.
\ee
Using \fref{riesz} and the H\"older inequality, it follows that
\bea \nonumber
\|R_{t,3}\|_{L^{q'}(\Rm^{d+1})} &\leq& C\| t^{-\eps} (1+|x|^{1+\alpha+\eps}) u_t \|_{L^{q'}(\Rm_t,L^{s}(\Rm_x^{d}))} \\\nonumber
&\leq& C\| t^{-\eps} u_t \|_{L^{q'}(\Rm_t,L^{r'}(\Rm_x^{d}))}\\ 
&\leq& C\left(\int_{-T}^T t^{-\eps q' p' } \right)^{1/(q'p')} \|\phi_t\|_{L^{r'}(\Rm^{d+1})}, \label{intT}
\eea
where $pq'=r'$, $p\in (1,\infty)$. In the second line above, we used the fact that $u_t$ has a compact support in order to control the $L^s_x$ norm of $u_t$ by the $L^{r'}_x$ norm, which is possible since $s<r'$. With $pq'=r'$ and $q'=\frac{s d}{d-\eps s}$, we find therefore
$$
p'=r' \left(\frac{d-\eps s }{r'(d-\eps s)-sd}\right)
$$
with the constraints
\be \label{eqp}
p=\frac{r'}{q'}=\frac{(d-\eps s)r'}{sd} > 1, \qquad i.e. \qquad 1<s<\frac{dr'}{d+\eps r'}.
\ee
For the first term of \fref{intT} to be finite, we need $\eps q'p'<1$, that is 
\be \label{est4}
\eps q' p'=\frac{\eps r'sd}{r'(d-\eps s)-sd}<1, \qquad i.e. \qquad sr'\eps (d+1)+sd<dr'.
\ee
When the latter inequality is satisfied, it implies the second inequality in \fref{eqp} that we thus ignore. We collect now the various constraints on the parameters and define $\eps$ and $s$ appropriately. The first inequality in \fref{const} becomes
$$
\frac{1}{2} < \frac{1}{s}-\frac{\eps}{d}<1.
$$
The second inequality holds since $s>1$ and $\eps>0$, and we therefore only keep the first one. 
% Let
% $$
% \eps=\frac{d(r'-s)}{sr'(d+1)}-\delta \leq 1.
% $$
% $$
% \frac{1}{2} \leq \frac{1}{s}\left(1-\frac{(r'-s)}{r'(d+1)}\right)+\delta
% $$
Defining $\delta$ by $r'=s+\delta$, with $0<\delta<r'-1$ since $r'>s>1$, \fref{est4} and the above inequality become
\be \label{const5}
r' (r'-\delta) \eps (d+1)<\delta d, \qquad \frac{1}{2} < \frac{1}{r'-\delta}-\frac{\eps}{d}.
\ee
We have in addition $r'\in (1,2]$ as well as
\be \label{const3} 
0<\eps\leq 1-\alpha, \qquad \alpha+\beta < \left(\frac{1}{s'}+\frac{\eps}{d}\right)(d+1)-\frac{d}{2}.
\ee
Since $r'\leq 2$ and $\delta>0$, $r'-\delta \leq 2-\delta<2$, the inequalities in \fref{const5} are satisfied as soon as
$$
4 (d+1) \eps \leq \delta d, \qquad \textrm{and} \qquad \frac{1}{2} <\frac{1}{2-\delta}-\frac{\eps}{d}, \qquad \textrm{that is for the latter} \qquad \frac{4\eps}{d+2\eps} < \delta.
$$
These latter inequalities are satisfied for instance when 
$$
\eps=\frac{\delta d}{4(d+1)} < \frac{1}{4} \qquad (\textrm{since} \quad\delta <r'-1\leq 1),
$$
and we verify that $\alpha+\eps \leq 1$ since $\alpha \leq 1/4$. It remains to choose $\delta$. We need to do it in such a way that $0<\delta<r'-1$ and such that the second inequality of \fref{const3} holds when \fref{mainconst} is verified. The latter condition can be expressed as the inequality, for $r \in [2,\infty)$,
$$
\frac{(d+1)}{r}\leq \left(\frac{1}{s'}+\frac{\eps}{d}\right)(d+1)+\alpha.
$$
With direct algebra and our current choice of $\eps$, this becomes
\be \label{const22}
\frac{\delta}{r'(r'-\delta)} \leq \frac{\eps}{d}+\frac{\alpha}{d+1}=\frac{\delta }{4(d+1)}+\frac{\alpha}{d+1}.
\ee
Since $r'>1$ and $r'-\delta>1$, we have $1<r'(r'-\delta)$ and \fref{const22} is satisfied whenever
$$
\delta \leq \frac{\delta }{4(d+1)}+\frac{\alpha}{d+1}, \qquad \textrm{that is} \qquad \delta \leq \frac{\alpha (4d+1)}{(d+1)(4d+3)}.
$$
Setting finally
$$
\delta= \min \left( \frac{\alpha (4d+1)}{(d+1)(4d+3)},\frac{r'-1}{2} \right),
$$
we verify that all constraints are satisfied. This concludes the estimation of the term $R_{t,3}$ and yields \fref{R3}. Regarding $R_{t,4}$, we use once more the Hardy-Littlewood-Sobolev inequality to obtain, for the same $q',\eps$ and $s$ as above and $s_0=\frac{sd}{d+s(\alpha+\eps)}<s$,
$$
\|R_{t,4}\|_{L^{q'}(\Rm^{d+1})} \leq  C\| t^{-\eps} (1+|x|) u_t \|_{L^{q'}(\Rm_t,L^{s_0}(\Rm_x^{d}))} \leq C\| t^{-\eps} u_t \|_{L^{q'}(\Rm_t,L^{s}(\Rm_x^{d}))}
$$
since $u_t$ has a compact support in $x$. It then suffices to proceed as for $R_{t,3}$. Finally, as already mentioned, the term $R_{t,2}$ is more regular and is treated with a similar and simpler analysis. This then enables us to estimate $I_0$ by
$$
\|I_0\|_{L^2(\Rm^d)} \leq C \|\phi_t\|_{L^{r'}(\Rm^{d+1})},
$$
under the condition \fref{mainconst}. We turn now to the term $I_2$.

\paragraph{The term $I_2$.} We start with the following lemma.
\begin{lemma} \label{com} Let $\varphi \in \calS(\Rm^{d})$, and let $\alpha,\beta>0$ with $\alpha+\beta<1$. Then, for any $\eps>0$, we have the estimate
$$
\left\| \big[1+|x|^\alpha,\Lambda_{\alpha+\beta}\big] \varphi \right \|_{L^2(\Rm^{d})} \leq C \| \varphi \|_{H^{\beta+\eps}(\Rm^d)}.$$
\end{lemma}
\begin{proof}
We proceed in the Fourier space. First, the Fourier transform of $|x|^\alpha$ reads, see \cite{gelfand}, section 3.3,
$$
\calF_x(|x|^\alpha)(\xi)= C_{\alpha,d} |\xi|^{-d-\alpha}, \qquad C_{\alpha,d}= 2^{\alpha+d} \pi^{\frac{d}{2}} \frac{\Gamma(\frac{\alpha+d}{2})}{\Gamma(-\frac{\alpha}{2})}.
$$
Therefore,
\bee
\calF_x \left(\big[1+|x|^\alpha,\Lambda_{\alpha+\beta}\big ]\varphi\right)(\xi)&=&C_{\alpha,d}\int_{\Rm^d} \frac{(1+|k|^2)^{(\alpha+\beta)/2}-(1+|\xi|^2)^{(\alpha+\beta)/2}}{|\xi-k|^{d+\alpha}} \calF_x \varphi(k) dk\\
&:=&F(\xi).
\eee
Since the function $ \Rm \ni u \mapsto (1+u^2)^{\gamma/2}$ is of H\"older regularity $\gamma$ for $\gamma \in (0,1)$, there exists a constant $C$ such that
$$
|(1+|k|^2)^{(\alpha+\beta)/2}-(1+|\xi|^2)^{(\alpha+\beta)/2}| \leq C||k|-|\xi||^{\alpha+\beta} \leq C |k-\xi|^{\alpha+\beta}.
$$ 
Hence, $F$ can be estimated by 
$$
|F(\xi)|\leq C \int_{\Rm^d} \frac{|\calF_x \varphi(k)| }{|\xi-k|^{d-\beta}} dk.
$$
Using the estimate \fref{riesz} on Riesz potentials, we find, together with the H\"older inequality for the second line,
\bee
\|F\|_{L^2(\Rm^d)} &\leq& C \| \calF_x \varphi\|_{L^{\frac{2d}{d+2\beta}}(\Rm^d)}\\
&\leq & C \|(1+|\xi|^{(\frac{\beta+\eps}{\beta})d})^{-1}\|_{L^1(\Rm^d)}  \| (1+|\xi|^{\beta+\eps})\calF_x \varphi\|_{L^{2}(\Rm^d)}\\[2mm]
&\leq & C \| (1+|\xi|^{\beta+\eps})\calF_x \varphi\|_{L^{2}(\Rm^d)}.
\eee
This concludes the proof of the lemma.
\end{proof}

\medskip

With the latter lemma at hand, the estimation of $I_2$ is now straightforward: we have, for any $\eps>0$,
\bee
\|I_2\|_{L^2(\Rm^d)}&=&\left\|\big[1+|x|^\alpha,\Lambda_{\alpha+\beta}\big] \int_{\Rm } dt e^{-i t \Delta} u_t \right\|_{L^2(\Rm^d)}\\
& \leq& C \left\| \Lambda_{\beta+\eps} \int_{\Rm } dt e^{-i t \Delta} u_t \right\|_{L^2(\Rm^d)}\\
&\leq & C \| \phi_t\|_{L^{r'}(\Rm^d)},
\eee
where we used Theorem \ref{thCS} in the last line with
$$
\beta+\eps < \frac{d+1}{r}-\frac{d}{2}, \qquad r'\in[1,2), \qquad \textrm{and} \qquad \beta+\alpha \leq \frac{1}{2}, \qquad r'=2.
$$
With for instance $\eps=\alpha$, the latter condition is implied by \fref{mainconst}. Note as well that the relation above implies $\beta \leq 1/2$, so that with the fact that $\alpha \leq 1/4$ according to the hypotheses of Lemma \ref{pr1}, we have $\alpha+\beta<1$ and Lemma \ref{com} can indeed be applied. 

In order to conclude the proof of Lemma \ref{pr1}, it remains to treat the term $I_1$, which is a direct consequence of Theorem \ref{thCS}. This ends the proof.

\section{Proof of Corollary \ref{cor1}} \label{proofcor}
As in the proof of Theorem \ref{th1}, we focus on $n_{\gamma_\beta}$, and consider a nonnegative finite rank ``smooth'' operator $\varrho$ verifying \fref{condrho}, the final estimates following by density. This will justify the formal calculations. In this setting,  we have in particular that $\Gamma (t) \in \calJ_1$ and $(\mu_j(t))^{1/2}\phi_j(t) \in \calS(\Rm^d)$ for all $t \in \Rm$, for $(\mu_j,\phi_j)_{p\in \Nm}$ the spectral decomposition of $\Gamma$. For $v\in C_0^\infty(\Rm^d)$, let then 
$$
n^v(t,x):=|v(x)|^2 n_{\Gamma}(t,x)=\sum_{j\in \Nm} \mu_j(t) |\phi^v_j(t,x)|^2, \qquad \phi^v_j:=v \phi_j.
$$
The proof consists in applying $D^\beta$ to $n^v$ and in using the estimates of Theorem \ref{th1}. This requires some care since there is no simple product rule for the fractional derivative. Let $\varphi \in \calS(\Rm^d)$. Then,
with the local representation of the fractional Laplacian \fref{locfrac} and the identity $|a|^2-|b|^2= 2\Re\, (a-b)(\overline{a}+\overline{b})$, we find, for $\beta \in (0,1/2]$,
\bee
D^\beta(|\varphi|^2)(x)&=&c_{d,\beta} \int_{\Rm^d} \frac{|\varphi(x)|^2-|\varphi(y)|^2}{|x-y|^{d+\beta}}dy\\
&=&2 c_{d,\beta} \Re\int_{\Rm^d} \frac{(\varphi(x)-\varphi(y))(\overline{\varphi(y)}-\overline{\varphi(x)}+2\overline{\varphi(x)} )}{|x-y|^{d+\beta}}dy\\[2mm]
&=&4 \Re \; \overline{\varphi} D^\beta(\varphi)(x)-2 c_{d,\beta} |\calD_{\beta/2}\varphi|^2,
\eee
where
$$
\calD_{\beta/2}(\varphi)=\left(\int_{\Rm^d} \frac{|\varphi(x)-\varphi(y)|^2}{|x-y|^{d+\beta}}dy\right)^{1/2}.
$$
Replacing $\varphi$ by $ (\mu_j(t))^{1/2}\phi_j^v(t)$, we obtain the following expression of $D^\beta n^v$:
$$
D^\beta n^v= 4 \Re \, \sum_{j\in \Nm} \mu_j\overline{\phi^v_j} D^\beta(\phi^v_j)-2 c_{d,\beta} \, \sum_{j\in \Nm} \mu_j |\calD_{\beta/2}(\phi^v_j)|^2.
$$
The first term of the r.h.s. can be thought of as the most singular term, and can be treated with Theorem \ref{th1} after some technicalities. An important point is the fact that we do not use the triangle inequality for the estimation, and therefore the fact that the $\phi_j$'s are orthogonal is directly exploited. The second term is more critical. To the best of our knowledge, there are not many ways to estimate $\calD_{\beta/2}(\phi^v_j)$, and one is given in \cite{stein1}, Theorem 1, and provides controls of the $L^p$ norm of $\calD_{\beta/2}(\phi_j^v)$ in terms of the $L^p$ norm of $\Lambda_{\beta/2}\phi^v_j$. This means that we need to use the triangle inequality at this stage. This does not imply that the orthogonality of the $\phi_p$'s is not used at all, it is when we invoke Theorem \ref{th1} when $r=2$ in the case $p>1$. These latter facts explain the limitation to the $L^1$ norm in time. The results could be extended to higher norms in time if we were able to replace $\Lambda_\beta$ by $\calD_{\beta}$  in the definition of $\gamma_\beta$, and prove an analog to Theorem \ref{th1}. This does not seem trivial since $\calD_{\beta}$ is not linear.

Using therefore the Cauchy-Schwarz and triangle inequalities, we find
\bea \nonumber
\|D^\beta n^v\|_{L^{\frac{d}{d-\beta}}(\Rm_x^d)}&\leq& C \left\| \left(\sum_{j\in \Nm} \mu_j |\phi^v_j|^2\right)^{1/2}\left(\sum_{j\in \Nm} \mu_j |D^\beta(\phi^v_j)|^2\right)^{1/2} \right\|_{L^{\frac{d}{d-\beta}}(\Rm^d_x)}\\
&& \qquad +C \sum_{j\in \Nm} \mu_j\big\|\calD_{\beta/2}(\phi_j^v) \big\|_{L^{\frac{2d}{d-\beta}}(\Rm^d_x)}^2. \label{Dn}
\eea
Thanks to \cite{stein1}, Theorem 1, we can control the term in $\calD_{\beta/2}(\phi_j^v)$ by (remark that obviously $2d/(d+\beta)<2d/(d-\beta)$),
$$
\|\calD_{\beta/2}(\phi_j^v) \big\|_{L^{\frac{2d}{d-\beta}}(\Rm^d_x)} \leq C \|\Lambda_{\beta/2} \phi_j^v\|_{L^{\frac{2d}{d-\beta}}(\Rm^d_x)} \leq C \|\Lambda_{\beta} \phi_j^v\|_{L^{2}(\Rm^d_x)},
$$
where we used the Sobolev embedding $H^{\beta/2,2}(\Rm^d) \subset L^{\frac{2d}{d-\beta}}(\Rm^d)$ for the second inequality. Furthermore, using the H\"older inequality and the embedding $H^{\beta,2}(\Rm^d) \subset L^{\frac{2d}{d-2\beta}}(\Rm^d)$ for the first term in the r.h.s of \fref{Dn}, we find
\bee
\|D^\beta n^v\|_{L^{\frac{d}{d-\beta}}(\Rm_x^d)}&\leq& C \sum_{j\in \Nm} \mu_j \left( \|\Lambda^\beta(\phi^v_j)\|^2_{L^2(\Rm_x^d)}+\|D^\beta(\phi^v_j)\|^2_{L^2(\Rm_x^d)}\right)\\
&\leq & C  \sum_{j\in \Nm} \mu_j \left( \|\phi^v_j\|^2_{L^2(\Rm^d_x)}+\|D^\beta(\phi^v_j)\|^2_{L^2(\Rm_x^d)} \right).
\eee
Above, we used \cite{stein1}, Theorem 2, in order to estimate $\Lambda_\beta$ in terms of $D^\beta$. We control now $D^\beta( v \phi_p)$ in terms of $v D^\beta(\phi_p)$. We have, for all $\varphi \in \calS(\Rm^d)$,
\bee
D^\beta(v \varphi)(x)&=&c_{d,\beta} \int_{\Rm^d} \frac{v(x)\varphi(x)-v(y)\varphi(y)}{|x-y|^{d+\beta}}dy\\
&=&c_{d,\beta} \int_{\Rm^d} \frac{v(x)(\varphi(x)-\varphi(y))+(v(x)-v(y))\varphi(y)}{|x-y|^{d+\beta}}dy\\
&=&v(x)D^\beta(\varphi)(x)+c_{d,\beta} \int_{\Rm^d} \frac{v(x)-v(y)}{|x-y|^{d+\beta}}\varphi(y) dy.
\eee
Supposing that the support of $v$ is compactly embedded in a bounded set $\Omega$, we write the last term above as
$$
c_{d,\beta} \int_{\Omega} \frac{v(x)-v(y)}{|x-y|^{d+\beta}}\varphi(y) dy+c_{d,\beta} v(x)\int_{\Omega^c} \frac{\varphi(y)}{|x-y|^{d+\beta}} dy.
$$
Hence,
\begin{align*}
& \sum_{j\in \Nm} \mu_j(t) \|D^\beta(\phi^v_j(t,\cdot))\|^2_{L^2(\Rm^d_x)}=\left\| \sum_{j\in \Nm} \mu_j(t) |D^\beta(\phi^v_j(t,\cdot))|^2 \right\|_{L^1(\Rm^d_x)}\\
&\hspace{1.8cm} \leq C \left\| \sum_{j\in \Nm} \mu_j(t) |v D^\beta(\phi_j(t,\cdot))|^2 \right\|_{L^1(\Rm^d_x)}+C\left\| N_1(t,\cdot) \right\|_{L^1(\Rm^d_x)}+C\left\| N_2(t,\cdot) \right\|_{L^1(\Rm^d_x)},
\end{align*}
where we have defined
\bee
N_1(t,x)&:=&\sum_{j\in \Nm} \mu_j(t)\left|\int_{\Omega} \frac{v(x)-v(y)}{|x-y|^{d+\beta}}\phi_j(t,y) dy \right|^2\\
 N_2(t,x)&:=& \sum_{j\in \Nm} \mu_j(t)\left|v(x)\int_{\Omega^c} \frac{\phi_j(t,y)}{|x-y|^{d+\beta}} dy \right|^2.
\eee
Writing the squares above as the product of two integrals, and using the Cauchy-Schwarz inequality lead to
$$
N_1(t,x) \leq \left|\int_{\Omega} \frac{|v(x)-v(y)|}{|x-y|^{d+\beta}} (n_{\Gamma}(t,y))^{1/2} dy \right|^2, \quad N_2(t,x) \leq |v(x)|^2 \left|\int_{\Omega^c} \frac{(n_{\Gamma}(t,y))^{1/2} }{|x-y|^{d+\beta}} dy \right|^2.
$$
Using the estimate \fref{riesz} on Riesz potentials and the fact that $|v(x)-v(y)|\leq C |x-y|$, we find, when $d\geq 2$,
\bee
\left\| N_1(t,\cdot) \right\|_{L^1(\Rm^d_x)} &\leq& C \int_{\Rm^d} \left|\int_{\Omega} \frac{(n_{\Gamma}(t,y))^{1/2} }{|x-y|^{d+\beta-1}} dy \right|^2 dx\\[2mm]
& \leq &C \|n_{\Gamma}(t,\cdot)\|_{L^{\frac{d}{d+2(1-\beta)}}(\Omega)} \leq C  \|n_{\Gamma}(t,\cdot)\|_{L^{1}(\Omega)}.
\eee
When $d=1$, \fref{riesz} cannot be used as above. We then control $|v(x)-v(y)|$ instead by $|v(x)-v(y)|\leq C |x-y|^{1/2}$, and since $\beta \in (0,\frac{1}{2}]$, we can apply \fref{riesz} and obtain
$$
\left\| N_1(t,\cdot) \right\|_{L^1(\Rm_x)} \leq C \|n_{\Gamma}(t,\cdot)\|_{L^{\frac{1}{1+2(1/2-\beta)}}(\Omega)} \leq C  \|n_{\Gamma}(t,\cdot)\|_{L^{1}(\Omega)}.
$$
Regarding $N_2$, for $x \in \textrm{supp } v \subset \subset \Omega$ and $y \in \Omega^c$, there exists a constant $C$ such that $|x-y|^{-1} \leq C (1+|y|)^{-1}$, and therefore
$$
\left\| N_2(t,\cdot) \right\|_{L^1(\Rm^d_x)} \leq C \|n_{\Gamma}(t,\cdot)\|_{L^{q}(\Rm^d_x)}, \qquad \forall q\in[1,\infty].
$$
It remains to bound $n_{\Gamma}$. Note that according to the last estimate, this has to be done over $\Rm^d$ and not just on a bounded domain, and this is a consequence of the non-locality of the fractional derivative. Estimating $n_{\Gamma}$ is not trivial unless $\varrho$ is trace class so that the triangle inequality can be used. When $\varrho$ is not trace class, the orthogonality of the eigenfunctions plays again a crucial role, and we then use the Strichartz estimates of \cite{FrankSabin}, Theorem 15, to obtain 
$$
\|n_{\Gamma} \|_{L^{\frac{p'}{d}}(-T,T,L^{\frac{p}{2-p}}(\Rm^d))} \leq C \|\varrho \|_{L^1(\Rm,\calJ_p)}, \qquad p \in [1,2d/(2d-1)].
$$
We verify indeed that our choice of parameters above satisfy the assumptions of the theorem of \cite{FrankSabin}, using in particular that $p'>2d$ when $\beta \geq 0$. Collecting all previous estimates, we find
$$
\|D^{\beta}(n^v) \|_{L^1(-T,T,L^{\frac{d}{d-\beta}}(\Rm^d))} \leq  C \left\| \sum_{j\in \Nm} \mu_j |v D^\beta(\phi_j)|^2 \right\|_{L^1((-T,T) \times \Rm^d)} + C \|\varrho \|_{L^1(\Rm,\calJ_p)}
$$
and conclude by controlling the first term in r.h.s. using Theorem \ref{th1} with $r=2$ and $\beta$ satisfying the conditions stated in the corollary.  This ends the proof. %This proves the result when the test function $\chi$ is positive. When it does not have a particular sign, we decompose it into positive and negative parts as $\chi=chi_+-\chi_-$. An inspection of the proof of the positive case shows that the only regularity assumption on $\chi$ that is used is that  This ends the proof.

{\footnotesize \bibliographystyle{siam}
  \bibliography{bibliography} }

 \end{document}